\documentclass[psamsfonts]{amsart}
\usepackage{amssymb,amsmath}
\usepackage{amsthm}
\usepackage{mathrsfs}
\usepackage{enumerate, indentfirst}
\usepackage[fleqn,tbtags]{mathtools}
\usepackage{courier}
\usepackage{hyperref}

\usepackage{xr}

\hypersetup{
    citecolor=black,%
    filecolor=black,%
    linkcolor=black,%
    pdffitwindow=true,%
    colorlinks=false,%
    unicode=true,
    }
\newtheoremstyle{mytheorem}%
{5pt}%
{3pt}%
{\itshape}%
{1pt}%
{\bf}%
{.}%
{.5em}%
{}%

\newtheoremstyle{myremark}%
{5pt}%
{3pt}%
{\upshape}%
{1pt}%
{\em}%
{.}%
{.5em}%
{}%

\newtheoremstyle{myexample}%
{5pt}%
{3pt}%
{\upshape}%
{1pt}%
{\bf}%
{.}%
{.5em}%
{}%

\theoremstyle{mytheorem}
\newtheorem{theorem}{Theorem}[section]

\newtheorem{lemma}[theorem]{Lemma}
\newtheorem{proposition}[theorem]{Proposition}
\newtheorem{corollary}[theorem]{Corollary}

\theoremstyle{myremark}

\newtheorem{remark}[theorem]{Remark}
\theoremstyle{myexample}

\numberwithin{equation}{section}
\makeatletter \renewenvironment{proof}[1][\proofname] {\par\pushQED{\qed}\normalfont\topsep6\p@\@plus6\p@\relax\trivlist\item[\hskip\labelsep\itshape #1\@addpunct{.}]\ignorespaces}{\popQED\endtrivlist\@endpefalse} \makeatother

\renewcommand{\phi}{\varphi}
\renewcommand{\theta}{\vartheta}
\renewcommand{\epsilon}{\varepsilon}

\DeclareMathOperator{\Sp}{Sp}

\newcommand{\pair}[2]{\left(\begin{array}{c}\!\!#1\!\!\\ \!\!#2\!\!\end{array}\right)}
\newcommand{\operator}[2]{\left(\begin{array}{cc}\!\! I &  #1\!\!\\ \!\! #2& I\!\!\end{array}\right)}

\DeclareMathOperator{\sform}{\mathfrak{s}}
\DeclareMathOperator{\tform}{\mathfrak{t}}

\DeclareMathOperator{\aform}{\mathfrak{a}}
\DeclareMathOperator{\bform}{\mathfrak{b}}
\newcommand{\dupN}{\mathbb{N}}

\newcommand{\seq}[1]{(#1_{n})_{n\in\dupN}}

\newcommand{\dupR}{\mathbb{R}}

\newcommand{\dom}{\operatorname{dom}}

\newcommand{\ran}{\operatorname{ran}}

\newcommand{\hil}{\mathscr{H}}

\newcommand{\kil}{\mathscr{K}}

\DeclarePairedDelimiterX\sip[2]{(}{)}{#1\,\delimsize\vert\,#2}
\DeclarePairedDelimiterX\siptilde[2]{(}{)_{\!_{\widetilde{A}}}}{#1\,\delimsize\vert\,#2}
\DeclarePairedDelimiterX\sipn[2]{(}{)_{\nu}}{#1\,\delimsize\vert\,#2}
\DeclarePairedDelimiterX\sipm[2]{(}{)_{\mu}}{#1\,\delimsize\vert\,#2}
\DeclarePairedDelimiterX\sips[2]{(}{)_{\sform}}{#1\,\delimsize\vert\,#2}
\DeclarePairedDelimiterX\sipt[2]{(}{)_{\tform}}{#1\,\delimsize\vert\,#2}
\DeclarePairedDelimiterX\set[2]{\{}{\}}{#1\,\delimsize\vert\,#2}
\DeclarePairedDelimiterX\dual[2]{\langle}{\rangle}{#1,#2}
\DeclarePairedDelimiterX\sipa[2]{(}{)_{\aform}}{#1\,\delimsize\vert\,#2}
\DeclarePairedDelimiterX\sipb[2]{(}{)_{\bform}}{#1\,\delimsize\vert\,#2}
\DeclarePairedDelimiterX\sipv[2]{(}{)_{v}}{#1\,\delimsize\vert\,#2}
\DeclarePairedDelimiterX\sipw[2]{(}{)_{w}}{#1\,\delimsize\vert\,#2}

\newcommand{\limn}{\lim\limits_{n\rightarrow\infty}}

\begin{document}
\title[On the sum between a closable operator...]{On the sum between a closable operator $T$ and a $T$-bounded operator}

\author[D. Popovici]{Dan Popovici}
\address{D. Popovici, Department of Mathematics, West University of Timi\c{s}oara, Bd. Vasile P\^{a}rvan nr. 4,
RO-300223 Timi\c{s}oara, Romania;}
\email{popovici@math.uvt.ro}

\author[Z. Sebesty\'en]{Zolt\'an Sebesty\'en}
\address{Z. Sebesty\'en, Department of Applied Analysis, E\"otv\"os L. University, P\'azm\'any P\'eter s\'et\'any 1/c., Budapest H-1117, Hungary; }
\email{sebesty@cs.elte.hu}

\author[Zs. Tarcsay]{Zsigmond Tarcsay}
\address{Zs. Tarcsay, Department of Applied Analysis, E\"otv\"os L. University, P\'azm\'any P\'eter s\'et\'any 1/c., Budapest H-1117, Hungary; }
\email{tarcsay@cs.elte.hu}

\keywords{Perturbation, $T$-boundedness, unbounded, closable, closed, symmetric, selfadjoint, essentially selfadjoint operator}
\subjclass[2010]{Primary  47A55, 47B25}

\begin{abstract}
We provide several perturbation theorems regarding closable operators on a real or complex Hilbert space. In particular we extend some classical results due to Hess--Kato, Kato--Rellich and W\"ust. Our
approach involves  ranges of matrix operators of the form $\operator{A}{-B}$.
\end{abstract}

\maketitle

\section{Introduction}

In the present paper we develop a perturbation theory of closable operators between Hilbert spaces. Operators we consider are (unless it is otherwise indicated) not necessarily densely defined
linear transformations and the Hilbert spaces are allowed to be either real or complex. The domain, kernel
and range of  an operator $A$ is denoted by $\dom A$, $\ker A$ and $\ran A$, respectively. The scalar product of each Hilbert space we encounter is denoted by the same symbol $\sip{\cdot}{\cdot}$
in the hope that we do not cause any confusion. As usual, $A^*$ stands for the adjoint operator of a densely defined operator $A$. A not necessarily densely defined operator $S$ is said to 
be symmetric if it satisfies
\begin{equation*}
    \sip{Sx}{y}=\sip{x}{Sy}, \qquad x,y\in \dom S,
\end{equation*}
and skew-symmetric if
\begin{equation*}
    \sip{Sx}{y}=-\sip{x}{Sy}, \qquad x,y\in \dom S.
\end{equation*}
For a densely defined $S$ the above relations mean that $S\subset S^*$ and $S\subset-S^*$, respectively. $S$ is called selfadjoint (resp., skew-adjoint) if $S=S^*$ (resp., $S=-S^*$).

If $S$ and $T$ are arbitrary operators with $\dom S\subseteq\dom T$ and such that there exist  $a,b\geq0$ satisfying
\begin{equation}\label{E:Tbounded}
    \|Sx\|^2\leq a\|Tx\|^2+b\|x\|^2,\qquad x\in\dom T,
\end{equation}
then $S$ is called $T$-bounded. The $T$-bound of $S$ in that case is defined to be the infimum of all nonnegative numbers $a$ for which a $b\geq0$ exists such that $\eqref{E:Tbounded}$ satisfies.
The notion of $T$-boundedness  is a useful tool  of the classical perturbation theory. For example, it appears as a basic condition of the  Hess--Kato \cite{HessKato}, Kato--Rellich \cite{rellich}
and W\"ust \cite{Wüst} perturbation theorems. The main goal of this note is to provide similar perturbation results under weaker conditions (see Theorem \ref{T:theorem12}, Corollary \ref{C:corollary18} and Corollary
\ref{C:corollary19} below). Our results involve ranges of $2$-by-$2$ matrix operators of the form $\operator{A}{-B}$. In our most recent works \cite{PoSe1,PoSe2} the
conditions provided by these matrices are replaced by similar conditions involving the operators $I+AB$ and $I+BA$.

\section{Closability of operators}

We start our discussion with an easy but useful result which in fact is a part of \cite[Lemma 3.1]{Popovici}. We present also its proof for the sake of the reader.
\begin{lemma}\label{L:lemma1}
    Let $S$ and $T$ be (possibly unbounded) linear operators between the Hilbert spaces $\hil$ and $\kil$, resp., $\kil$ and $\hil$. If
    \begin{enumerate}[\upshape (a)]
      \item $\ran S=\kil$,
      \item $\overline{\ran T}=\hil$,
      \item $\sip{Sh}{k}=\sip{h}{Tk}$, $h\in\dom S, k\in\dom T$,
    \end{enumerate}
    then $T$ is automatically densely defined such that $T^*=S$.
\end{lemma}
\begin{proof}
     In order to prove that $T$ is densely defined, let $v\in\dom T^{\perp}$. Consider $h\in\dom S$ such that $Sh=v$. As, for any $k\in\dom T$
     \begin{equation*}
        0=\sip{v}{k}=\sip{Sh}{k}=\sip{h}{Tk},
     \end{equation*}
     we deduce that $h\in\ran T^{\perp}=\{0\}$. Hence $v=Sh=0$, as required. It is obvious by (c) that $S\subset T^*$. For the converse inclusion we should show that $\dom T^*\subseteq\dom S$. With this aim let us take $k\in\dom T^*$ and $u\in\dom S$ such that $Su=T^*k$. As $S\subset T^*$, we have $T^*k=T^*u$ and hence
     \begin{equation*}
        k=k-u+u\in\ker T^*+\dom S\subseteq \dom S.
     \end{equation*}
     Here, for the last equality we used the fact that $\ker T^*=\ran T^{\perp}=\{0\}$.
\end{proof}
\begin{theorem}\label{T:theorem2}
    Let $\hil_1, \hil_2, \hil_3$ and $\hil_4$ be Hilbert spaces and consider the (not necessarily densely defined) linear operators $A:\hil_1\to\hil_4$, $B:\hil_2\to\hil_3$, $C:\hil_3\to\hil_2$ and $D:\hil_4\to\hil_1$. If
    \begin{enumerate}[\upshape (a)]
      \item $\ran\operator{A}{-B}=\hil_4\times\hil_3$,
      \item $\overline{\ran\operator{-C}{D}}=\hil_2\times\hil_1$,
      \item $\sip{Ax_1}{x_4}-\sip{x_1}{Dx_4}=\sip{Bx_2}{x_3}-\sip{x_2}{Cx_3}(=0)$, $x_1\in\dom A, x_2\in\dom B, x_3\in\dom C, x_4\in\dom D,$
    \end{enumerate}
    then $C$ and $D$ are densely defined such that $D^*=A$ and $C^*=B$.
\end{theorem}
\begin{proof}
    We may use Lemma \ref{L:lemma1} for the Hilbert spaces $\hil=\hil_2\times\hil_1$, $\kil=\hil_4\times\hil_3$ and the operators
    \begin{equation*}
        S=\operator{A}{-B},\qquad T=\operator{-C}{D}.
    \end{equation*}
    We firstly observe that
    \begin{align*}
        \sip[\bigg]{\operator{A}{-B}\pair{x_2}{x_1}}{\pair{x_4}{x_3}}&=\sip{x_2+Ax_1}{x_4}+\sip{-Bx_2+x_1}{x_3}\\
        &=\sip{x_2}{x_4}+\sip{Ax_1}{x_4}-\sip{Bx_2}{x_3}+\sip{x_1}{x_3}
    \end{align*}
    and that
    \begin{align*}
        \sip[\bigg]{\pair{x_2}{x_1}}{\operator{-C}{D}\pair{x_4}{x_3}}&=\sip{x_2}{x_4-Cx_3}+\sip{x_1}{Dx_4+x_3}\\
        &=\sip{x_2}{x_4}+\sip{x_1}{Dx_4}-\sip{x_2}{Cx_3}+\sip{x_1}{x_3}.
    \end{align*}
    The two quantities are equal if and only if (c) holds true. We deduce therefore by Lemma \ref{L:lemma1} that
    \begin{equation*}
        \overline{\dom \operator{-C}{D}}=\hil_4\times\hil_3\qquad\textrm{and}\qquad \operator{-C}{D}^*=\operator{A}{-B}.
    \end{equation*}
    As $\dom T=\dom D\times \dom C$, we conclude that $C$ and $D$ are densely defined operators. In addition, since
    \begin{equation*}
        \operator{-C}{D}^*=\operator{D^*}{-C^*},
    \end{equation*}
    it follows that $D^*=A$ and $C^*=B$.
\end{proof}
\begin{remark}
    We mention here the elementary fact that for given two operators $S$ and $T$, $\operator{S}{-T}$ has full
(resp., dense) range if and only if $\operator{-S}{T}$ has full (dense) range. The proof is left to the reader.
\end{remark}
\begin{theorem}\label{T:theorem4}
     Let $A$ and $B$ be (possibly unbounded) linear operators between the Hilbert spaces $\hil$ and $\kil$, resp., $\kil$ and $\hil$. The following assertions are equivalent:
     \begin{enumerate}[\upshape (i)]
       \item $B$ is a densely defined closable operator such that $B^*=A$;
       \item $A$ is a densely defined closed operator such that
       \begin{equation*}
       \sip{Ax}{y}=\sip{x}{By}, \qquad x\in\dom A, y\in\dom B,
       \end{equation*}
 and that $\overline{\ran\operator{A}{-B}}=\kil\times\hil$.
     \end{enumerate}
\end{theorem}
\begin{proof}
    We firstly observe that $\operator{A}{-A^*}$ is surjective under assumption (ii). Indeed, as $A$ is densely defined and closed, we have by von Neumann's formulae
    \begin{align*}
        \kil\times \hil&=\set{(y,-A^*y)}{y\in\dom A^*}\oplus\set{(Ax,x)}{x\in\dom A}\\
                       &=\set{(y+Ax,x-A^*y)}{x\in\dom A, y\in\dom A^*}=\ran \operator{A}{-A^*}.
    \end{align*}
     Furthermore, by Theorem \ref{T:theorem2} we conclude that
    \begin{equation*}
        \overline{\dom B}=\hil\qquad \textrm{and}\qquad B^*=A.
    \end{equation*}
    Hence (ii) implies (i). Conversely, assume that $A,B$ fulfill (i). Obviously, $A$ is closed (being the adjoint of $B$) and densely defined (as $B$ is closable).
    Our only claim is therefore to show that  $\overline{\ran\operator{B^*}{-B}}=\kil\times\hil$ holds for any closable operator between $\kil$ and $\hil$. Equivalently,
    \begin{equation*}
        \ker\operator{-B^*}{B^{**}}=\{0\}.
    \end{equation*}
    Indeed, for if $x\in\dom B^{*}, y\in\dom B^{**}$ then
    \begin{gather*}
        \left\|\operator{-B^*}{B^{**}}\pair{y}{x}\right\|^2=\|y-B^*x\|^2+\|B^{**}y+x\|^2\\
        =\|y\|^2+\|B^*x\|^2+\|x\|^2+\|B^{**}y\|^2\\
        -\sip{B^*x}{y}-\sip{y}{B^*x}+\sip{B^{**}y}{x}+\sip{x}{B^{**}y}\\
        =\|y\|^2+\|B^*x\|^2+\|x\|^2+\|B^{**}y\|^2\\
        \geq \left\|\pair{y}{x}\right\|^2.
    \end{gather*}
    (Here we used the well known identity $B^*=(B^{**})^*$). Hence (i) implies (ii).
\end{proof}
\begin{remark}
    The surjectivity of the matrix operator $\operator{A}{-A^*}$
(for a given densely defined operator $A$) has been observed in \cite[Propostion 2.3 (b)]{Popovici} to be equivalent with the property of $A$
to be closed. In \cite[Propostion 2.3 (a)]{Popovici} it was shown, with a different proof, that $\operator{A}{-A^*}$ has dense range even without the
assumption on the closability of $A$.
\end{remark}
\begin{remark}
    In the previous result we obtain the same conclusion if we replace $I$ by $\alpha I$ for a certain/for any $\alpha \in\dupR\setminus\{0\}$.
\end{remark}
\begin{remark}
    We also mention, without proof, that for given   be densely defined linear operators $A, B$ between the Hilbert spaces $\hil$ and $\kil$, resp., $\kil$ and $\hil$, the conditions
    \begin{enumerate}[\upshape (i)]
      \item $\overline{\ran \operator{A}{B}}=\kil\times\hil,$
      \item $1\notin\Sp_p(B^*A^*)$,
      \item $1\notin\Sp_p(A^*B^*)$,
    \end{enumerate}
    are equivalent. This result is obtained in  \cite[Proposition 2.4]{Popovici}.
\end{remark}
Theorem \ref{T:theorem4} provides a useful characterization of essentially selfadjointness, cf. also \cite[Corollary 4.8]{Popovici}:
\begin{corollary}\label{C:corollary6}
    Let $S$ be a linear operator acting in the (real or complex) Hilbert space $\hil$. The following statements are equivalent:
    \begin{enumerate}[\upshape (i)]
      \item $S$ is essentially selfadjoint;
      \item $S$ is densely defined symmetric such that $\overline{\ran\operator{S}{-S}}=\hil\times\hil$.
    \end{enumerate}
\end{corollary}
\begin{proof}
    Theorem \ref{T:theorem4} can be applied for $A=S^{**}$ and $B=S$.
\end{proof}
\section{Perturbation theorems for closable operators}
This section is devoted to the perturbation theory of linear operators. More precisely, we deal with perturbations $S+T$ of closable, selfadjoint and essentially selfadjoint operators $T$ by
a $T$-bounded operator $S$ and we provide some generalizations of important results due to Hess--Kato \cite{HessKato}, Kato--Rellich \cite{rellich}, and W\"ust \cite{Wüst}.
\begin{theorem}\label{T:theorem10}
    Let $W,Z$ be linear operators in the Hilbert space $\hil$ with $\dom Z\subseteq\dom W$ and so that $Z$ and $W|_{\dom Z}$ are skew-symmetric:
    \begin{align*}
        \sip{Zh}{k}+\sip{h}{Zk}&=0,\\
        \sip{Wh}{k}+\sip{h}{Wk}&=0,
    \end{align*}
    for $h,k\in\dom Z$. Assume furthermore that $W$ and $Z$ fulfill each of the following conditions:
    \begin{enumerate}[\upshape (a)]
      \item $Z$ is closable (or $W|_{\dom Z}$ is closable);
      \item $T:=I+Z$ has dense range;
      \item $\|Wh\|^2\leq\|Th\|^2$ for all $h\in\dom T=\dom Z.$
    \end{enumerate}
    Then $T+W$ has dense range in $\hil$.
\end{theorem}
\begin{proof}
    Let $k\in\ran (T+W)^{\perp}$, that is to say,
    \begin{equation*}
        \sip{k}{(T+W)h}=0\qquad \textrm{for all $h\in\dom Z$}.
    \end{equation*}
    By b) we may choose a sequence $\seq{h}$ in $\dom Z$ such that
    \begin{equation*}
        \|k-Th_n\|=\|k-(I+Z)h_n\|\to0,\qquad n\to\infty.
    \end{equation*}
    Here $\|T(h_n-h_m)\|^2=\|Z(h_n-h_m)\|^2+\|h_n-h_m\|^2$ by skew-symmetry, whence we conclude that $\seq{h}$ and $\seq{Zh}$ also converge:
    \begin{equation*}
        h_n\to \widetilde{h},\qquad Zh_n\to z.
    \end{equation*}
    Let us observe that for $h\in\dom Z$
    \begin{align*}
        \|k+Wh\|^2&=\|k\|^2+\|Wh\|^2+\sip{k}{Wh}+\sip{Wh}{k}\\
                  &=\|k\|^2+\|Wh\|^2-\sip{k}{Th}-\sip{Th}{k}\\
                  &\leq\|k\|^2+\|Th\|^2-\sip{k}{Th}-\sip{Th}{k}\\
                  &=\|k-Th\|^2.
    \end{align*}
    Hence if we take $h=h_n$ and let $n\to\infty$ we see that $\|Wh_n+k\|^2\to0$. In particular,
    \begin{equation*}
        (W+T)h_n\to0.
    \end{equation*}
    Next we claim that
    \begin{equation}\label{C:claim1}
        \|k-(W+T)h\|^2=\|k\|^2+\|(Z+W)h\|^2+\|h\|^2,\qquad h\in\dom Z.
    \end{equation}
    Indeed, since $k\in\ran(W+T)^{\perp}$ and by skew-symmetry we have
    \begin{align*}
        \|k-(W+&T)h\|^2=\|k\|^2+\|(I+Z+W)h\|^2\\
                      &=\|k\|^2+\|h\|^2+\|(Z+W)h\|^2+\sip{h}{(Z+W)h}+\sip{(Z+W)h}{h}\\
                      &=\|k\|^2+\|(Z+W)h\|^2+\|h\|^2.
    \end{align*}
    Letting, as before, $h=h_n$ and passing to limit we deduce by \eqref{C:claim1} that
    \begin{equation*}
        \|k\|^2=\limn\|k-(W+T)h_n\|^2=\|k\|^2+\|z-k\|^2+\|\widetilde{h}\|^2,
    \end{equation*}
    so $\widetilde{h}=0$ and $k=z$, that is,
    \begin{equation*}
        h_n\to0, \qquad Zh_n\to k.
    \end{equation*}
    Finally we deduce by the fact that $Z$ is closable that $k=0$. Consequently, $\ran(W+T)^{\perp}=\{0\}$, as required.
\end{proof}
The key to the upcoming perturbation theorems is the next corollary (combined with Theorem \ref{T:theorem4}):
\begin{corollary}\label{C:corollary11}
    Let $A,R$ and $B,S$ be linear operators between the Hilbert spaces $\hil$ and $\kil$, respectively $\kil$ and $\hil$, with $\dom A\subseteq\dom R$ and $\dom B\subseteq\dom S$ so that
    \begin{align*}
        \sip{Ah}{k}=\sip{h}{Bk},\qquad h\in\dom A, k\in\dom B,\\
        \sip{Rh}{k}=\sip{h}{Sk},\qquad h\in\dom R, k\in\dom S.
    \end{align*}
 Assume furthermore that $A,B,R,S$ fulfill each of the following conditions:
    \begin{enumerate}[\upshape (a)]
      \item $A,B$ are closable (or $R|_{\dom A}$ and $S|_{\dom B}$ are closable);
      \item $\operator{A}{-B}$ has dense range in $\hil\times\kil$;
      \item $\|Rh\|^2\leq\|Ah\|^2+\|h\|^2$ for all $h\in\dom A,$
      \item $\|Sk\|^2\leq\|Bk\|^2+\|k\|^2$ for all $h\in\dom B$.
    \end{enumerate}
    Then
    \begin{equation*}
        \overline{\ran\operator{A+R}{-(B+S)}}=\kil\times\hil.
    \end{equation*}
\end{corollary}
\begin{proof}
The proof is a simple application of Theorem \ref{T:theorem10} for
\begin{equation*}
    Z=\begin{pmatrix}
        0 & A \\
        -B & 0 \\
      \end{pmatrix}
    \qquad \textrm{and } \qquad W=\begin{pmatrix}
        0 & R \\
        -S & 0 \\
      \end{pmatrix}.
\end{equation*}
\end{proof}
The next result is an immediate generalization of the Hess--Kato \cite{HessKato} perturbation theorem:
\begin{theorem}\label{T:theorem12}
    Let $R$ and $B,S$ be linear operators from $\hil$ to $\kil$, and $\kil$ to $\hil$, respectively, satisfying the following conditions:
    \begin{enumerate}[\upshape (a)]
      \item $\dom B^*\subseteq\dom R$ and $\dom B\subseteq\dom S$;
      \item $B$ is densely defined and closable;
      \item $B^*+R$ is closed;
      \item $\sip{Rh}{k}=\sip{h}{Sk}$ for all $h\in\dom B^*$ and $k\in\dom B$;
      \item $\|Rh\|^2\leq\|B^*h\|^2+\|h\|^2$ for $h\in\dom B^*$;
      \item $\|Sk\|^2\leq\|Bk\|^2+\|k\|^2$ for $k\in\dom B$.
    \end{enumerate}
    Then $B+S$ is (densely defined and) closable such that
    \begin{equation*}
        (B+S)^*=B^*+R
    \end{equation*}
\end{theorem}
\begin{proof}
    Since $B$ is densely defined it follows that
    \begin{equation*}
        \overline{\ran \operator{B^*}{-B}} =\kil\times \hil.
    \end{equation*}
    Moreover, $B^*+R$ is densely defined and closed by assumption, such that
    \begin{equation*}
        \sip{(B^*+R)h}{k}=\sip{h}{(B+S)k},\qquad h\in\dom B^*, k\in\dom B.
    \end{equation*}
    Consequently, due to the previous Corollary, $B^*+R$ and $B+S$ fulfill the conditions of Theorem \ref{T:theorem4}: $B+S$ is closable and $(B+S)^*=B^*+R$.
\end{proof}
\begin{remark}\label{R:remark13}
    A simple condition which guarantees the closedness of $B^*+R$ is that $R$ is assumed to be $B^*$-bounded, by $B^*$-bound less than 1, see \cite[\S 3, Theorem 4.2]{Birman}.
\end{remark}
\begin{corollary}\label{C:corollary14}
    Let $B,S$ be  linear operators from $\hil$ to $\kil$ satisfying the following conditions:
    \begin{enumerate}[\upshape (a)]
      \item $\dom B\subseteq\dom S$ and $\dom B^*\subseteq\dom S^*$;
      \item $B$ is densely defined and closable;
      \item $B^*+S^*$ is closed;
      \item $\|Sh\|^2\leq\|Bh\|^2+\|h\|^2$ for $h\in\dom B$;
      \item $\|S^*k\|^2\leq\|B^*k\|^2+\|k\|^2$ for $k\in\dom B^*$;
    \end{enumerate}
    Then $B+S$ is (densely defined and) closable such that
    \begin{equation*}
        (B+S)^*=B^*+S^*.
    \end{equation*}
\end{corollary}
\begin{proof}
    Use Theorem \ref{T:theorem12} for $R:=S^*$.
\end{proof}
The following proposition extends a classical result due to Hess--Kato \cite{HessKato}; cf. also \cite[\S 3, Theorem 4.3]{Birman}:
\begin{proposition}
    Let $S$ and $T$ be densely defined linear operators between the Hilbert spaces $\hil$ and $\kil$ satisfying
   \begin{enumerate}[\upshape (a)]
     \item $T$ is closable;
     \item $\dom T\subseteq\dom S$ and $\dom T^*\subseteq\dom S^*$;
     \item $\|Sh\|^2\leq \|Th\|^2+\|h\|^2$ for all $h\in\dom T$;
     \item $\|S^*k\|^2\leq q\|T^*k\|^2+\|k\|^2$ for all $k\in\dom T^*$ and for some $0\leq q<1$.
   \end{enumerate}
   Then $B+S$ is closable and $(B+S)^*=B^*+S^*$.
\end{proposition}
\begin{proof}
    Use Corollary \ref{C:corollary14} and Remark  \ref{R:remark13}.
\end{proof}
The next corollary is a generalized W\"ust perturbation theorem (cf. \cite{Wüst})
for essentially selfadjoint operators in real or complex Hilbert spaces (cf. \cite{Weidmann}[\S 4.1, Theorem 4.6]; to see also [9]).
\begin{corollary}\label{C:corollary18}
    Let $S,T$ be linear operators acting on a Hilbert space $\hil.$ Assume that
    \begin{enumerate}[\upshape (a)]
     \item $S$ is essentially selfadjoint with $\dom S\subseteq\dom T$;
     \item $\|Th\|^2\leq \|Sh\|^2+\|h\|^2$ for all $h\in\dom S$;
     \item $\sip{Th}{k}=\sip{h}{Tk}$ for all $h,k\in\dom S$ (that is, the restriction of $T$ to $\dom S$ is symmetric).
   \end{enumerate}
   Then $S+T$ is essentially selfadjoint.
\end{corollary}
\begin{proof}
    We show that $A:=B:=S$ and $R:=S:=T$ fulfill each of the conditions (a)-(d) of  Corollary \ref{C:corollary11}.
    Indeed, $\dom S\subseteq\dom T$, and, being $S,T$ symmetric,
    \begin{equation*}
        \sip{Sh}{k}=\sip{h}{Sk}\qquad \sip{Th}{k}=\sip{h}{Tk},\qquad h,k\in\dom S.
    \end{equation*}
    Furthermore, $S$ is essentially selfadjoint, thus $\overline{\ran \operator{S}{-S}}=\hil\times\hil$ by  Corollary \ref{C:corollary6}. Finally,
    $\|Th\|^2\leq \|Sh\|^2+\|h\|^2$ for all $h\in\dom T$, by assumption. Hence
    \begin{equation*}
        \ran\operator{S+T}{-(S+T)}=\hil\times\hil.
    \end{equation*}
    Since $S+T$ is symmetric and densely defined, Corollary \ref{C:corollary6} applies.
\end{proof}
An immediate consequence of the previous result is the following
\begin{corollary}\label{C:corollary17}
    Let $S,T$ be linear operators acting on a Hilbert space $\hil.$ Assume that
    \begin{enumerate}[\upshape (a)]
     \item $S$ is essentially selfadjoint with $\dom S\subseteq\dom T$;
     \item $\|Th\|^2\leq \|Sh\|^2+\|h\|^2$ for all $h\in\dom S$;
     \item $\sip{Th}{k}=\sip{h}{Tk}$ for all $h,k\in\dom S$;
     \item $S+T$ is closed.
   \end{enumerate}
   Then $S+T$ is selfadjoint.
\end{corollary}
We close our paper with a generalized version of the classical Kato--Rellich \cite{rellich} perturbation theorem; cf. also \cite[\S 5, Theorem 5.28]{Weidmann} and \cite{Characterization} :
\begin{corollary}\label{C:corollary19}
    Let $S,T$ be linear operators acting on a Hilbert space $\hil.$ Assume that
    \begin{enumerate}[\upshape (a)]
     \item $S$ is selfadjoint with $\dom S\subseteq\dom T$;
     \item $\|Th\|^2\leq q\|Sh\|^2+\|h\|^2$ for all $h\in\dom S$ with some $0\leq q<1$;
     \item $\sip{Th}{k}=\sip{h}{Tk}$ for all $h,k\in\dom S$.
   \end{enumerate}
   Then $S+T$ is selfadjoint.
\end{corollary}
\begin{proof}
    Apply Corollary \ref{C:corollary17} and Remark \ref{R:remark13}.
\end{proof}
\bibliographystyle{abbrv}

\end{document}